\documentclass[11pt]{article}
\usepackage[centertags]{amsmath}
\usepackage{amsfonts}
\usepackage{amssymb}
\usepackage{amsthm,color}
\usepackage{newlfont}
\usepackage{bbm}

\newtheorem{Theorem}{Theorem}[section]

\newtheorem{Lemma}[Theorem]{Lemma}

\newtheorem{Remark}[Theorem]{Remark}

\newtheorem{Hypothesis}{Hypothesis}

\numberwithin{equation}{section}

\begin{document}

\def\le{\left}
\def\r{\right}
\def\cost{\mbox{const}}
\def\a{\alpha}
\def\d{\delta}
\def\ph{\varphi}
\def\e{\epsilon}
\def\la{\lambda}
\def\si{\sigma}
\def\La{\Lambda}
\def\B{{\cal B}}
\def\A{{\mathcal A}}
\def\L{{\mathcal L}}
\def\O{{\mathcal O}}
\def\bO{\overline{{\mathcal O}}}
\def\F{{\mathcal F}}
\def\K{{\mathcal K}}
\def\H{{\mathcal H}}
\def\D{{\mathcal D}}
\def\C{{\mathcal C}}
\def\M{{\mathcal M}}
\def\N{{\mathcal N}}
\def\G{{\mathcal G}}
\def\T{{\mathcal T}}
\def\R{{\mathbb R}}
\def\I{{\mathcal I}}

\def\bw{\overline{W}}
\def\phin{\|\varphi\|_{0}}
\def\s0t{\sup_{t \in [0,T]}}
\def\lt{\lim_{t\rightarrow 0}}
\def\iot{\int_{0}^{t}}
\def\ioi{\int_0^{+\infty}}
\def\ds{\displaystyle}
\def\pag{\vfill\eject}
\def\fine{\par\vfill\supereject\end}
\def\acapo{\hfill\break}

\def\beq{\begin{equation}}
\def\eeq{\end{equation}}
\def\barr{\begin{array}}
\def\earr{\end{array}}
\def\vs{\vspace{.1mm}   \\}
\def\rd{\reals\,^{d}}
\def\rn{\reals\,^{n}}
\def\rr{\reals\,^{r}}
\def\bD{\overline{{\mathcal D}}}
\newcommand{\dimo}{\hfill \break {\bf Proof - }}
\newcommand{\nat}{\mathbb N}
\newcommand{\E}{\mathbb E}
\newcommand{\Pro}{\mathbb P}
\newcommand{\com}{{\scriptstyle \circ}}
\newcommand{\reals}{\mathbb R}

\newcommand{\red}[1]{\textcolor{red}{#1}}

\def\Amu{{A_\mu}}
\def\Qmu{{Q_\mu}}
\def\Smu{{S_\mu}}
\def\H{{\mathcal{H}}}
\def\Im{{\textnormal{Im }}}
\def\Tr{{\textnormal{Tr}}}
\def\E{{\mathbb{E}}}
\def\P{{\mathbb{P}}}
\def\span{{\textnormal{span}}}
\title{Large deviations for the dynamic $\Phi^{2n}_d$ model}
\author{Sandra Cerrai\thanks{Partially supported by the NSF grant DMS 1407615.}\\
\normalsize University of Maryland, \\
United States
\and
Arnaud Debussche\thanks{Partially supported by the French government thanks to the
ANR program Stosymap and the ``Investissements d'Avenir" program ANR-11-LABX-0020-01. }\\
\normalsize IRMAR, ENS Rennes, CNRS, UBL, \\ France}
\date{}

\date{}

\maketitle

\begin{abstract}
We are dealing with the validity of a large deviation principle for  a class of reaction-diffusion equations with polynomial non-linearity, perturbed by a Gaussian random forcing. We are here interested in the regime where  both the strength of the noise and its correlation are vanishing, on a length scale $\e$ and $\d(\e)$, respectively, with $0<\e,\d(\e)<<1$. We prove that, under the assumption that $\e$ and $\d(\e)$ satisfy a suitable scaling limit, a large deviation principle holds in the space of continuous trajectories with values both  in the space of square-integrable functions and in Sobolev spaces of negative exponent. Our result is valid, without any restriction on the degree of the polynomial nor on the space dimension.
\end{abstract}

\section{Introduction}
\label{sec1}

We are  dealing here with the equation
\begin{equation}
\label{eq}
\begin{cases}
\ds{\partial_t u(t,\xi)=\Delta u(t,\xi)+f(u(t,\xi))+\sqrt{\e}\, \xi^\d(t,\xi),\ \ \ t>0,\ \ \ \ \ \ \xi \in\,D,}\\
\vs
\ds{u(0,\xi)=x(\xi),\ \ \ \xi \in\,D,\ \ \ \ \ \ \ u(t,\xi)=0,\ \ \ t\geq 0,\ \ \ \ \ \xi \in\,\partial D,}
\end{cases}
\end{equation}
defined in a bounded smooth domain $D\subset \mathbb{R}^d$, with $d\geq 1$.
The nonlinearity $f$ is given by the polynomial
\[f(r)=-r^{2n+1}+\la_1\,r+\la_2,\ \ \ \ r \in\,\mathbb{R},\]
for some $n \in\,\nat$ and $\la_1, \la_2 \in\,\mathbb{R}$.
The forcing term $\xi^\d(t,\xi)$ is a zero mean space-time Gaussian noise, white in time and colored in space, with correlation of order $\d$, and  $\e>0$ is the parameter that measures the intensity of the noise. 

If $\d>0$, then,  by using classical arguments in the theory of SPDEs, it is possible to show that, for every fixed $\e>0$, equation \eqref{eq} is globally well posed (for a proof, see e.g. \cite[Theorem 7.19]{dpz1}).
On the other hand, if the space dimension $d$ is bigger than $1$, and  the Gaussian noise is white, both in time and in space, (that is $\d=0$)   the well-posedness of equation \eqref{eq} is a problem  and a proper renormalization of the non-linear term $f$
is required. In case of 
space dimension $d=2$, this renormalization is realized through the
Wick ordering (to this purpose, see \cite{DPD}, \cite{JLM} and \cite{MW}). In case $d=3$ and $f$ is a polynomial of degree $3$, the proof of the well-posedness of the problem requires a considerably more complicated renormalization of the non-linearity (see \cite{H}, and also \cite{MW1} for the global well-posedness). Nothing of what we have mentioned applies in dimension $d=4$ and higher.

Here, we are  interested in the validity of a large deviation principle for equation \eqref{eq}, when both $\e$ and $\d$ go to zero. In \cite{cf11} it has been studied this problem when first $\e\to 0$ and then $\d\to 0$, in the  case $f$ is a Lipschitz-continuous nonlinearity, without any restriction on the dimension. It has been proved that the action functional $I^\d_T$, that describes the large deviation principle for the family $\{u_{\d}^\e\}_{\e>0}$ in the space $C([0,T];L^2(D))$, is $\Gamma$-convergent, as $\d\downarrow 0$,  to the functional
\begin{equation}
\label{ult}
I_T(u)=\frac 12 \int_0^T|\partial_t u(t)-\Delta u(t)-f(u(t))|^2_{L^2(D)}\,dt.\end{equation}
The functional $I_T$ corresponds to the large deviation  action functional for equation \eqref{eq}, in case of space-time white noise, when well-posedness is a challenge. In particular, the $\Gamma$-convergence of $I^\d_T$ to $I_T$ has allowed to obtain the  converge of the quasi-potential and, as a consequence, the approximation of the expected exit times and exit places from suitable functional domains by the solution of equation \eqref{eq}.

In \cite{HW},  Hairer and Weber have studied the large deviation principle for equation \eqref{eq}, with $f(r)=-r^3+\la_1 r$, in dimension $d=2, 3$, under the assumption that $\d=\d(\e)$. By using the recently developed theory of regularity structures, they have proved the validity of a large deviation principle for the family of random variables $\{u_{\e}\}_{\e>0}$, where $u_\e=u^\e_{\d(\e)}$,   in  case 
\begin{equation}
\label{cond}
\lim_{\e\to 0} \d(\e)=0.\end{equation}
 Actually, they have  proved that if, in addition to \eqref{cond}, the following conditions hold
\begin{equation}
\label{ult3}
\lim_{\e\to 0}\e\,\log\d(\e)^{-1}=\rho \in\,[0,\infty),\ \ \ \text{for}\ d=2, \ \ \ \ \ \ \ \lim_{\e\to 0}\e\,\d(\e)^{-1}=\rho \in\,[0,\infty),\ \ \ \text{for}\ d=3,\end{equation}
then the family $\{u_{\e}\}_{\e>0}$ satisfies a large deviation principle in $C([0,T],C^\eta(D))$, where $C^\eta(D)$ is some space of functions of negative regularity in space, with respect to the action functional
\[I_T^\rho(u)=\frac 12\int_0^T|\partial _t u-\Delta u+c_\rho\,u+u^3|_{L^2(D)}^2\,dt.\]
Here $c_\rho$ is some explicitly given constant, depending on $\rho$ and $d$, and  such that $c_0=-\la_1$.

In \cite{HW}, Hairer and Weber have also considered the renormalized equation 
\[\left\{
\begin{array}{l}
\ds{\partial _t u(t,\xi)=\Delta u(t,\xi)+(c+3\,\e\,c_{\d(\e)}^{(1)}-9\,\e^2\,c_{\d(\e)}^{(2)})\,u(t,\xi)-u^3(t,\xi)+\sqrt{\e}\, \xi^{\d(\e)}(t,\xi),}\\
\vs
\ds{u(0,\xi)=u_0(\xi),\ \ \ \xi \in\,D,}
\end{array}\r.\]
where $c_{\d(\e)}^{(1)}$ and $c_{\d(\e)}^{(2)}$ are the constants that  arise from the renormalization procedure. They have proved that if in this case \eqref{cond} holds, then the family of solutions $\{u_{\e}\}_{\e>0}$ satisfies a large deviation principle in $C([0,T],C^\eta(D))$,  with action functional $I^0_T$.

Hairer and Weber's proof of the large deviation
principle relies strongly on the understanding of the renormalized equation even for
the schemes without renormalization. In particular, in \cite{HW} they claim that it is not clear whether a large
deviations principle holds in higher dimensions, even in the regime $\e<<\d(\e)^{d-2}$. 

In the present paper, by using the so called {\em weak convergence approach} to large deviations (see \cite{BDM}), we extend Hairer and Weber's result to polynomials $f$ of any degree and to any space dimension $d\geq 2$. Actually, we prove that  the family of solutions $\{u_{\e}\}_{\e>0}$ of equation \eqref{eq} satisfies a large deviation principle in $C([0,T];H^{-s}(D))$, for every $s>0$, with respect to the action functional $I_T$ defined in \eqref{ult}, under the assumption that $\d=\d(\e)$ satisfies condition \eqref{cond} and (in case of periodic boundary conditions)
\[\lim_{\e\to 0} \e\,\log \d(\e)^{-1}=0,\ \ \ \ \ \text{if}\ d=2,\]
and
\begin{equation}
\label{ult2}
\lim_{\e\to 0} \e\,\d(\e)^{-(d-2)}=0,\ \ \ \ \ \text{if}\ d\geq 3.\end{equation}   

Moreover, we prove the validity of a large deviation principle in $C([0,T];L^2(D))$, with respect to the same action functional $I_T$, under the more restrictive assumption that
\begin{equation}
\label{ult1000}
\lim_{\e\to 0} \e\,\d(\e)^{-\eta}=0,
\end{equation}
for some $\eta>d-2$.
In fact, in the present paper we  consider Dirichlet boundary conditions in a  general smooth bounded domain $D$ and, in this case,   scalings   \eqref{ult2} and \eqref{ult1000} become slightly different (see Theorem \ref{teo5.1} for the precise statement).

\section{Notations}
\label{sec2}
Let $D$ be a bounded domain in $\mathbb{R}^d$, having smooth boundary. In what follows, we shall denote by $H$ the Hilbert space $L^2(D)$, endowed with the usual scalar product 
\[\langle x,y\rangle_H=\int_D x(\xi) y(\xi)\,d\xi,\] 
 and the corresponding  norm $|\cdot|_H$. Moreover, we shall denote by $E$ the Banach space $C(\bar{D})$, endowed with the supremum norm 
 \[|x|_E=\sup_{x \in\,\bar{D}}|x(\xi)|,\]
 and the duality $\langle \cdot,\cdot\rangle _{E,E^\star}$.
 For any $p \in\,[1,\infty]\setminus \{2\}$, the norms in $L^p(D)$  will be  denoted by $|\cdot|_p$ and the duality between $L^p(D)$ and $L^q(D)$, with $p^{-1}+q^{-1}=1$, will be denoted by $\langle \cdot,\cdot\rangle_{p,q}$.
  
 Next, for any $x \in\,E$, we denote
\[M_x=\le\{\,\xi \in\,\bar{D}\,:\, |x(\xi)|=|x|_E\,\r\}.\]
Moreover, if $x\neq 0$, we set
\[\mathcal{M}_x=\le\{\,\d_{x,\xi} \in\,E^\star\,;\,\xi \in\,M_x\,\r\},\]
where $\d_{x,\xi}$ is the element of the dual $E^\star$ defined by
\[\langle \d_{x,\xi},y\rangle_{E,E^\star}=\frac{x(\xi)y(\xi)}{|x|_{E}},\ \ \ \ \ y \in\,E.\]
For $x=0$, we set
\[\mathcal{M}_0=\le\{\,h \in\,E^\star\,:\ |h|_{E^\star}=1\,\r\}.\]
Clearly, we have
\[\mathcal{M}_x\subseteq \partial |x|_E:=\le\{\,h
\in\,E^\star\,;\,|h|_{E^\star}=1,\
\le<h,x\r>_{E,E^\star}=|x|_E\,\r\},\]
for every $x \in\,E$, and, due to the characterization of $\partial |x|_E$, it is possible to show that if $\# M_x=1$, then $\mathcal{M}_x= \partial |x|_E$. In particular,    
if $u:[0,T]\to E$ is any differentiable mapping, then
\begin{equation}
\label{dersub}
\frac d{dt}^- |u(t)|_E\leq \le<u^\prime(t),\d\r>_{E,E^\star},
\end{equation}
for any $t \in\,[0,T]$ and $\d \in\,\mathcal{M}_{u(t)}$ (for all details we refer e.g. to \cite[Appendix D]{dpz1} and also to \cite[Appendix A]{tesi}).

\bigskip

   In what follows we shall denote by $A$ the realization 
in $H$ of the Laplace operator $\Delta$,
endowed with Dirichlet boundary conditions.
That is
\[D(A)=\le\{\,u \in\,W^{2,2}(D)\ :\ u(x)=0,\ x \in\,\partial D\,\r\},\ \ \ \ \ \ Au=\Delta u.\]
In fact, with the same arguments that we will use in the case of Dirichlet boundary conditions, we can also treat Neumann or periodic boundary conditions.

It is possible to check (see e.g. \cite{dav} for all details and
 proofs) that  $A$ is a non-positive and self-adjoint
operator in $H$, which generates an analytic semigroup $e^{tA}$ with
dense domain.
In \cite[Theorem 1.4.1]{dav} it is proved that  the space
$L^1(D)\cap L^\infty(D)$ is invariant under $e^{tA}$, so
that $e^{tA}$ may be extended to a non-negative one-parameter
contraction semigroup $T_{p}(t)$ on $L^p(D)$, for all $1\leq
p\leq \infty$. These semigroups are strongly continuous for $1\leq
p<\infty$ and are consistent, in the sense that
$T_{p}(t)u=T_{q}(t)u$, for all $u \in\,L^p(D)\cap
L^q(D)$. This is why we shall denote all $T_{p}(t)$ by
$e^{tA}$.  Finally, if we consider the part
of $A$ in the space of continuous functions $E$, it generates an
analytic semigroup which has no dense domain in general (it
clearly depends on the boundary conditions).

The semigroup $e^{tA}$ is compact on $L^p(D)$ for
all $1\leq p\leq \infty$ and $t>0$. The spectrum $\{-\a_{k}\}_{k
\in\,\nat}$ of $A$ is independent of $p$ and $e^{tA}$ is analytic
on $L^p(D)$, for all $1\leq p\leq \infty$. 
Moreover, there exists $c>0$ such that
\begin{equation}
\label{rd22}
c^{-1}\, k^{\frac 2d}\leq \a_k\leq c\, k^{\frac 2d},\ \ \ \ k \in\,\mathbb{N}.
\end{equation}

In what follows, for every $s>0$, we denote by $H^{-s}(D)$ the closure of $H$ with respect to the norm
\[|x|_{H^{-s}(D)}^2=\sum_{k=1}^\infty x_k^2 \,\a_k^{-s}.\]

Concerning the
complete orthonormal system of eigenfunctions $\{e_{k}\}_{k
\in\,\nat}$, in case $D=[0,L]^d$, we have $\sup_{k \in\,\mathbb{N}}|e_k|_\infty<\infty$. In case of a general bounded  domain $D$ in $\mathbb{R}^d$, with $d>1$, having a smooth boundary, we have that there exists some $c=c(D)>0$ such that
\[|e_k|_\infty \leq c\,\a_k^{\frac{d-1}4},\ \ \ \ \ k \in\,\mathbb{N}\]
(for a proof see \cite{greiser}, where the estimate above is proved for $d$-dimensional compact manifold with boundary). In particular, due to \eqref{rd22}, we have
\[|e_k|_\infty \leq c\,k^{\frac{d-1}{2d}},\ \ \ \ \ k \in\,\mathbb{N}.\]
Thus, in what follows, we will assume the following condition
\begin{Hypothesis}
\label{H1}
There exist $\a=\a(d)\geq 0$ and $c>0$ such that
\begin{equation}
\label{rd1}
|e_k|^2_\infty \leq c\,k^{\frac{\a}d} ,\ \ \ \ k \in\,\mathbb{N}.\end{equation}
\end{Hypothesis}

 \bigskip  
   
Now,   for every $x:D\to \mathbb{R}$ we shall denote
\[F(x)(\xi)=f(x(\xi)),\ \ \ \ \xi \in\,D,\]
where 
\[f(r)=-r^{2n+1}+\la_1\,r+\la_2,\ \ \ \ r \in\,\mathbb{R},\]
for some $n \in\,\nat$ and $\la_1, \la_2 \in\,\mathbb{R}$. It is immediate to check that $F$ maps $E$ into $E$ continuously, is locally Lipschitz continuous   and 
\[|F(x)|_E\leq c\le(|x|_E^{2n+1}+1\r),\ \ \ \ \ x \in\,E.\] Moreover, for every $x, y \in\, E$ and $\d \in\,\mathcal{M}_{x-y}$, we have
\begin{equation}
\label{dissE}
\langle F(x)-F(y),\d\rangle_{E,E^\star}\leq c\,|x-y|_E,\ \ \ \ \ x,y \in\,E.
\end{equation}
It is also possible to check that, if 
 we  denote
\[p_n=2(n+1),\ \ \ \ q_n=\frac{2(n+1)}{2n+1},\]
then $F$ maps $L^{p_n}(D)$ into $L^{q_n}(D)$ and
for every $x, y \in\,L^{p_n}(D)$ we have
\[\begin{array}{l}
\ds{|F(x)-F(y)|_{q_n}^{q_n}\leq c\int_D |x(\xi)-y(\xi)|^{q_n}\le(|x(\xi)|^{2n}+|y(\xi)|^{2n}+1\r)^{q_n}\,d\xi}\\
\vs
\ds{\leq c\,|x-y|_{p_n}^{q_n}\le(|x|_{p_n}+|y|_{p_n}+1\r)^{2nq_n}.}
\end{array}\]
This implies that for every $x, y \in\,L^{p_n}(D)$
\begin{equation}
\label{rd10}
|F(x)-F(y)|_{q_n}\leq c\,|x-y|_{p_n}\le(|x|_{p_n}^{2n}+|y|_{p_n}^{2n}+1\r).
\end{equation}
In particular, we get
\begin{equation}
\label{fin4}
|F(x)|_{q_n}\leq c\le(|x|_{p_n}^{2n+1}+1\r),\ \ \ \  x \in\,L^{p_n}(D).\end{equation}
Moreover, there exists some constant $c>0$ such that for every $r,s \in\,\mathbb{R}$
\[(f(r)-f(s))(r-s)\leq -c\,|r-s|^{p_n}+\la_1\,|r-s|^2,\]
and this implies that for every $x,y \in\,L^{p_n}(D)$, it holds
\begin{equation}
\label{rd3}
\langle F(x)-F(y),x-y\rangle_{q_n,p_n}\leq -c\,|x-y|_{p_n}^{p_n}+\la_1\,|x-y|_H^2.\end{equation}

In what follows, for every $N>0$ we shall define
\[f_N(r)=\le\{\begin{array}{ll}
\ds{f(r),}  & \ds{|r|\leq N,}\\
& \vs
\ds{f(N r/|r|),}  &  \ds{|r|>N.}
\end{array}\r.\]
and we shall denote by $F_N$ the composition operator associated with $f_N$. As $f_N:\mathbb{R}\to \mathbb{R}$ is Lipschitz continuous and bounded, the mapping $F_N:E\to E$ is Lipschitz-continuous and bounded. For every $M\geq N$, we have
\begin{equation}
\label{local}
|x|_E\leq N\Longrightarrow F_M(x)=F_N(x)=F(x).
\end{equation}
Moreover, it is possible to verify that for every $N>0$ and  $\d \in\,\mathcal{M}_{x-y}$
\begin{equation}
\label{dissEn}
\langle F_N(x)-F_N(y),\d_{x-y}\rangle_{E,E^\star}\leq c\,|x-y|_E,\ \ \ \ x,y \in\,E,
\end{equation}
for some constant $c$ independent of $N$.

\section{The model}
\label{sec3}

As we mentioned in the introduction, we are dealing here with the equation
\begin{equation}
\label{model}
\begin{cases}
\ds{\partial_t u(t,\xi)=\Delta u(t,\xi)+f(u(t,\xi))+\sqrt{\e}\,\partial_t w^\d(t,\xi),\ \ \ t>0,\ \ \ \ \ \ \xi \in\,D}\\
\vs
\ds{u(0,\xi)=x(\xi),\ \ \ \xi \in\,D,\ \ \ \ \ \ \ u(t,\xi)=0,\ \ \ t\geq 0,\ \ \ \ \ \xi \in\,\partial D.}
\end{cases}
\end{equation}
Concerning the random perturbation $w^\d(t)$, we assume that for every $\d>0$ it is a cylindrical Wiener process in $L^2(D)$, white in time and colored in space, with covariance
\[Q_\d=\le(I+\d\,\sqrt{-A}\r)^{-2\beta},\]
for some $\beta=\beta(d)\geq 0$, depending on the space dimension $d$. This means that $w^\d(t)$ can be represented as
\[w^\d(t)=\sum_{k=1}^\infty \la_k(\d) e_k\,\beta_k(t),\ \ \ \ \ t\geq 0,\]
where $\{\beta_k(t)\}_{k \in\,\mathbb{N}}$ is a sequence of independent standard Brownian motions defined on a stochastic basis $(\Omega,\mathcal{F}, \{\mathcal{F}_t\}_{t\geq 0},\mathbb{P})$, $\{e_k\}_{k \in\,\mathbb{N}}$ is the complete orthonormal systen of $L^2(D)$ that diagonalizes $A$ (see  Section \ref{sec2}) and
\begin{equation}
\label{rd31}
\la_k(\d)=\le(1+\d\,\sqrt{\a_k}\r)^{-\beta},\ \ \ \ \ k \in\,\mathbb{N}.\end{equation}

\begin{Hypothesis}
\label{H2}
For every $d>1$, we assume 
\begin{equation}
\label{rd35}
\beta=\beta(d)>\frac{d-2+\a}2,\end{equation}
where $\a=\a(d)$ is the non-negative constant introduced in Hypothesis \ref{H1}. 
\end{Hypothesis}

With the notation introduced in Section \ref{sec2}, for every $\e, \d>0$
equation \eqref{model} can be rewritten as the following abstract evolution equation
\begin{equation}
\label{abstract}
du(t)=\le[A u(t)+F(u(t))\r]\,dt+\sqrt{\e}\, dw^\d(t),\ \ \ \ u(0)=x.
\end{equation}
Due to Hypothesis \ref{H2}, for every $\e, \d>0$ the linear problem 
\[dz(t)=A z(t)\,dt+\sqrt{\e}\,dw^\d(t),\ \ \ \ z(0)=0,\]
 admits a unique mild solution $z^\e_\d$ belonging to $L^p(\Omega;C([0,T];E))$, for every $p\geq1$ and $T>0$.
Therefore, as proved in \cite[Theorem 7.19]{dpz1}, for any initial condition $x \in\,H$, equation \eqref{abstract} admits a unique mild solution $u_\d^\e \in\,L^p(\Omega;C([0,T];H)\cap L^{2(n+2)}((0,T)\times D))$, for every $p\geq 1$ and $T>0$.

\section{The skeleton equation}
\label{sec4}

We are here interested in the study of the well-posedness of the following deterministic problem
\begin{equation}
\label{skel}
\frac{du}{dt}(t)=A u(t)+F(u(t))+\varphi(t),\ \ \ \ u(0)=x,
\end{equation}
where the control  $\varphi$ is taken in $L^2(0,T;H)$ and the initial condition $x$ in $H$.

We recall that a function $u$ in $C([0,T];X)$ is a mild solution to equation \eqref{skel} if
\[u(t)=e^{t A}x+\int_0^t e^{(t-s)A} F(u(s))\,ds+\int_0^t e^{(t-s)A} \varphi(s)\,ds,\]
(here we denote by $X$ either $H$, or $E$, or $L^p(D)$, for $p\geq 1$).

\begin{Theorem}
\label{teo4.1}
For every $T>0$ and for every $x \in\,H$ and $\varphi \in\,L^2(0,T;H)$, there exists a unique mild solution $u^{x,\varphi}$  to equation \eqref{skel} in $C([0,T];H)\cap L^{p_n}((0,T)\times D)$.
Moreover
\begin{equation}
\label{rd20}
|u^{x,\varphi}|_{C([0,T];H)}+|u^{x,\varphi}|_{L^{p_n}((0,T)\times D)}\leq c_T\le(1+|x|_H+|\varphi|_{L^2(0,T;H)}\r).
\end{equation}

\end{Theorem}

\begin{proof}
 For every $N>0$, we introduce the approximating problem
\begin{equation}
\label{skel-n}
\frac{du}{dt}(t)=A u(t)+F_N(u(t))+\varphi(t),\ \ \ \ u(0)=x.
\end{equation}
As $F_N:E\to E$ is Lipschitz continuous, if $x \in\,E$ and $\varphi \in\,L^2(0,T;E)$  there exists a unique mild solution $u_N\in\,C([0,T];E)$. In case we want to emphasize the dependence of $u_N$  on the initial condition  $x$ and the control $\varphi$, we will denote it by $u_N^{x,\varphi}$.

Now, according to \eqref{dersub} and \eqref{dissEn}, for every $\d(t,N) \in\,\mathcal{M}_{u_N(t)}$ we have
 \[\begin{array}{l}
 \ds{\frac{d}{dt}^-|u_N(t)|_E\leq \langle A u_N(t),\d(t,N)\rangle_{E, E^\star}+\langle F_N( u_N(t))-F_N(0),\d(t,N)\rangle_{E, E^\star}}\\
 \vs
 \ds{+\langle F_N(0),\d(t,N)\rangle_{E, E^\star}+\langle \varphi(t),\d(t,N)\rangle_{E, E^\star}\leq c\,|u_N(t)|_E+|\varphi(t)|_E+\la_2,}\\
  \end{array}
 \]
so that
 \[|u_N(t)|_E\leq c_T\le(|x|_E+|\varphi|_{C([0,T];E)}+\la_2\r),\ \ \ \ t \in\,[0,T].\]
According to \eqref{local}, this means in particular that if we fix 
\[\bar{N}>c_T\le(|x|_E+|\varphi|_{C([0,T];E)}+\la_2\r),\] and define
\[u^{x,\varphi}(t)=u_{\bar{N}}^{x,\varphi}(t),\ \ \ \ t \in\,[0,T],\]
the function $u^{x,\varphi}$ is a mild solution to  problem \eqref{skel}. Moreover, $u^{x,\varphi}$ is the unique mild solution. Actually, if $v_1$ and $v_2$ are two mild solutions in $C([0,T];E)$ and $\rho=v_1-v_2$, due to \eqref{dissE}, for every $\d \in\,\mathcal{M}_{\rho(t)}$ we have
\[\begin{array}{l}
 \ds{\frac{d}{dt}^-|\rho(t)|_E\leq \langle A \rho(t),\d\rangle_{E, E^\star}+\langle F( v_1(t))-F(v_2(t)),\d\rangle_{E, E^\star}\leq c\,|\rho(t)|_E,}
\end{array}
 \]
and, as $\rho(0)=0$, we can conclude that $v_1(t)-v_2(t)=\rho(t)=0$, for every $t \in\,[0,T]$.
 
Now, if $x \in\,H$ and $\varphi \in\,L^2(0,T;H)$, let $\{x_k\}_{k \in\,\nat}\subset E$ and $ \{\varphi_k\}_{k \in\,\nat}\subset L^2(0,T;E)$ be two sequences such that
\begin{equation}
\label{rd5}
\lim_{k\to\infty}|x_k-x|_H+|\varphi_k-\varphi|_{L^2(0,T;H)}=0.\end{equation}
If we fix $k,h \in\,\mathbb{N}$ and define $\rho:=u^{x_k,\varphi_k}-u^{x_h,\varphi_h}$, we have that $\rho$ is a mild solution to the problem 
\[\frac{d\rho}{dt}(t)=A\rho(t)+\le[F(u^{x_k,\varphi_k}(t))-F(u^{x_h,\varphi_h}(t))\r]+\le[\varphi_k(t)-\varphi_h(t)\r],\ \ \ \ \rho(0)=x_k-x_h.\]
Therefore, due to \eqref{rd3}, we have
\[\begin{array}{l}
\ds{\frac 12 \frac d{dt}|\rho(t)|_H^2\leq \langle A\rho(t),\rho(t)\rangle_H+\langle F(u^{x_k,\varphi_k}(t))-F(u^{x_h,\varphi_h}(t)),\rho(t)\rangle_{q_n, p_n}}\\
\vs
\ds{+\langle \varphi_k(t)-\varphi_h(t),\rho(t)\rangle_H\leq -c\,|\rho(t)|_{p_n}^{p_n}+c\,|\rho(t)|_H^2+|\varphi_k(t)-\varphi_h(t)|_H^2.}
\end{array}\]
This implies that
\begin{equation}
\label{rd11}
|\rho(t)|^2_H+\int_0^t|\rho(s)|_{p_n}^{p_n}\,ds\leq c_T\le(|x_k-x_h|^2_H+|\varphi_k-\varphi_h|^2_{L^2(0,T;H)}\r).\end{equation}
In particular, due to \eqref{rd5}, we have
\[\lim_{k, h\to\infty} |u^{x_k,\varphi_k}-u^{x_h,\varphi_h}|_{C([0,T];H)}+|u^{x_k,\varphi_k}-u^{x_h,\varphi_h}|_{L^{p_n}((0,T)\times D)}=0,\]
so that the sequence $\{u^{x_k,\varphi_k}\}_{k \in\,\mathbb{N}}$ converges in $C([0,T];H)\cap L^{p_n}((0,T)\times D)$ to some $u^{x,\varphi}$, that satisfies estimate \eqref{rd20}. 

Thus, in order to conclude the proof of the present theorem, we have to show that $u^{x,\varphi} $ is a mild solution to equation \eqref{skel}.
 For every $k \in\,\mathbb{N}$, we have 
\begin{equation}
\label{rd14}
 u^{x_k,\varphi_k}(t)=e^{ t A}x_k+\int_0^t e^{(t-s)A} F(u^{x_k,\varphi_k}(s))\,ds+\int_0^t e^{(t-s)A} \varphi_k(s)\,ds.\end{equation}
According to \eqref{rd10}, we have
\[\begin{array}{l}
\ds{\le|\int_0^t e^{(t-s) A}\le[F(u^{x,\varphi}(s))-F(u^{x_k,\varphi_k}(s))\r]\,ds\r|_{q_n}}\\
\vs
\ds{\leq c\int_0^t |F(u^{x,\varphi}(s))-F(u^{x_k,\varphi_k}(s))|_{q_n}\,ds}\\
\vs
\ds{\leq c\int_0^t |u^{x,\varphi}(s)-u^{x_k,\varphi_k}(s)|_{p_n}\le(|u^{x,\varphi}(s)|^{2n}_{p_n}+
|u^{x_k,\varphi_k}(s)|_{p_n}^{2n}+1\r)\,ds,}
\end{array}\]
Therefore, since
 both $u^{x,\varphi}$ and $u^{x_k,\varphi_k}$ satisfy estimate \eqref{rd20}, we get
\begin{equation}
\label{rd13}
\begin{array}{l}
\ds{\int_0^T\le|\int_0^t e^{(t-s) A}\le[F(u^{x,\varphi}(s))-F(u^{x_k,\varphi_k}(s))\r]\,ds\r|_{q_n}\,dt}\\
\vs
\ds{\leq \le(\int_0^T|u^{x,\varphi}(s)-u^{x_k,\varphi_k}(s)|_{p_n}^{p_n}\,ds\r)^{\frac 1{p_n}}\le(\int_0^T \le(|u^{x,\varphi}(s)|^{2nq_n}_{p_n}+
|u^{x_k,\varphi_k}(s)|_{p_n}^{2n q_n}\r)\,ds+1\r)^{\frac 1{q_n}}}\\
\vs
\ds{\leq c_T(x,\varphi)\,
|u^{x,\varphi}-u^{x_k,\varphi_k}|_{L^{p_n}((0,T)\times D)}.}
\end{array}\end{equation}

Moreover, since we have
\[\sup_{t\in\,[0,T]}\le|\int_0^t e^{(t-s)A}\le[\varphi(s)-\varphi_k(s)\r]\,ds\r|_H\leq c_T |\varphi-\varphi_k|_{L^2(0,T;H)},\]
and
\[\sup_{t\in\,[0,T]}\le|e^{tA}(x-x_k)\r|_H\leq |x-x_k|_H,\]
due to \eqref{rd5} and \eqref{rd13} we can take the limit, as $k\uparrow \infty$, in both sides of \eqref{rd14} with respect to the  $L^1(0,T;L^{q_n}(D))$-norm and we get that $u^{x,\varphi}$ satisfies the equation
\[ u^{x,\varphi}(t)=e^{ t A}x+\int_0^t e^{(t-s)A} F(u^{x,\varphi}(s))\,ds+\int_0^t e^{(t-s)A} \varphi(s)\,ds.\]  

Finally, as any solution $u^{x,\varphi}$ satisfies estimate \eqref{rd20}, uniqueness follows.  

\end{proof}

\section{The large deviation result}
\label{sec5}

In Section \ref{sec3} we have seen that for every $\e, \d>0$  and every initial condition $x \in\,H$, equation \eqref{abstract} admits a unique mild solution $u^\e_\d \in\,L^p(\Omega;C([0,T];H)\cap L^{2(n+1)}((0,T)\times D))$. Here and in what follows, we shall  assume that $\d=\d(\e)>0$, for every $\e>0$, with
\[\lim_{\e\to 0}\d(\e)=0.\] 
Our purpose is proving the validity of a large deviation principle in the space $C([0,T];H^{-s}(D))$, for $s>0$, and in the space $C([0,T];H)$, as $\e\to 0$,  for the family of random variables $\{u_\e\}_{\e>0}$, where $u_\e=u^\e_{\d(\e)}$, for every $\e>0$. If we want to emphasize the dependence of $u_\e$ from its initial condition, we  denote it by $u^x_\e$.

\begin{Theorem}
\label{teo5.1}
Let Hypotheses \ref{H1} and \ref{H2} be satisfied and assume that 
\[\lim_{\e\to 0} \d(\e)=0.\]
If 
 \begin{equation}
\label{rd46}
\lim_{\e\to 0}\,\e\,\La(\d(\e))=0,\end{equation}
where
\[\La(\d):=\begin{cases} 
\log \d^{-1},  &  \text{if}\  \a=0 \ \text{and}\ d=2,\\
\vs
\d^{-(d-2+\a)},  &  \text{otherwise,}
\end{cases}\]
then,  for every initial condition $x \in\,H$ and for every $s>0$, the family of random variables $\{u^x_\e\}_{\e>0}$ satisfies a large deviation principle in $C([0,T];H^{-s}(D))$, with action functional
\begin{equation}
\label{rd51}
I_T(u)=\frac 12 \int_0^T\le|u^\prime(t)-Au(t)-F(u(t))\r|_H^2\,dt.
\end{equation}
Moreover, if there exists $\gamma>d-2+\a$ such that
\begin{equation}
\label{rd5050}
\lim_{\e\to 0} \e\,\d(\e)^{-\gamma}=0,
\end{equation}
then the family  $\{u^x_\e\}_{\e>0}$ satisfies a large deviation principle in $C([0,T];H)$, with respect to the same action functional $I_T$.
\end{Theorem}

As we have already done in our previous paper \cite{cd16}, where we have studied an analogous problem for the $2$-dimensional stochastic Navier-Stokes equation with periodic boundary conditions, we will prove Theorem \ref{teo5.1} by using the weak convergence approach to large deviations, as developed in \cite{BDM} in the case of SPDEs. To this purpose, we first introduce some notation and then we give two conditions that, in view of what proved in \cite{BDM}, imply the validity of the Laplace principle for the family $\{u_\e\}_{\e>0}$, with respect to the action functional $I_T$,  in the spaces $C([0,T];H^{-s}(D))$ and $C([0,T];H)$, depending on the different scaling conditions between $\e$ and $\d(\e)$ (see  \eqref{rd46} and \eqref{rd5050}). 

\medskip

In Theorem \ref{teo4.1} we have shown that, for every predictable process $\varphi(t)$ in $L^2(\Omega\times [0,T];H)$, the problem
\begin{equation}
\label{rd50}
\frac{du}{dt}(t)=Au(t)+F(u(t))+\varphi(t),\ \ \ \ \ u(0)=x,
\end{equation} 
admits a unique mild solution $u^{x, \varphi} \in\,C([0,T];H)\cap L^{2(n+1)}((0,T)\times D)$. 
By combining together the proof of Theorem \ref{teo4.1} with \cite[proof of Theorem 7.19]{dpz1}, it is possible to prove that for every fixed $\e>0$ the problem
\begin{equation}
\label{rd49}
du(t)=\le[Au(t)+F(u(t))+Q_{\d(\e)}\varphi(t)\r]\,dt+\sqrt{\e}\,dw^{\d(\e)}(t),\ \ \ \ \ u(0)=x,
\end{equation} 
admits a    unique mild solution  $u^{x, \varphi}_\e \in\,L^2(\Omega;C([0,T];H)\cap L^{2(n+1)}((0,T)\times D))$.

\begin{quotation}
\item[{\em Condition 1.}] If $I_T$ is the functional defined in \eqref{rd51}, the level sets $\le\{I_T(u)\leq r\r\}$ are compact in $C([0,T];H)$, for every $r\geq 0$.

\item[{\em Condition 2.}]
For every fixed $T>0$ and $\gamma>0$, let us define
\[\mathcal{A}_{T}^\gamma:=\le\{ u \in\,L^2(\Omega\times [0,T];H)\ \text{predictable}\,:\, \int_0^T|u(s)|_H^2\,ds\leq \gamma,\ \mathbb{P}-\text{a.s.}\r\}.\]
If the family $\{\varphi_\e\}_{\e>0}\subset \mathcal{A}_T^\gamma$  converges in distribution, as $\e\downarrow 0$, to some $\varphi \in\,\mathcal{A}_T^\gamma$, in the space $L^2(0,T;H)$, endowed with the weak topology, then the family $\{u^{x,\varphi_\e}_\e\}_{\e>0}$ converges in distribution to $u^{x,\varphi}$, as $\e\downarrow 0$, in the space $C([0,T];H^{-s}(D))$ or $C([0,T];H)$, depending if condition \eqref{rd46} or condition \eqref{rd5050} are satisfied, respectively.

\end{quotation}

As we already mentioned, in \cite{BDM} it is proved that if Condition 1 and Condition 2 hold, then the family of random variables $\{u_\e\}_{\e>0}$ satisfies a large deviation principle in the space $C([0,T];H)$, with respect to the action functional $I_T$ defined in \eqref{rd51}. This means that Theorem \ref{teo5.1} follows, once we prove that  Condition 1 and Condition 2 are both satisfied.

Condition 1 follows if we can prove that the mapping 
\[\varphi \in\,L^2(0,T;H)\mapsto u^\varphi \in\,C([0,T];H),\]
is continuous, when $L^2(0,T;H)$ is endowed with the weak topology and $C([0,T];H)$ is endowed with the strong topology.

As far as Condition 2 is concerned, we use the Skorohod theorem and rephrase such a condition in the following way.
Let $(\bar{\Omega}, \bar{\mathcal{F}},\bar{\mathbb{P}})$ be a probability space and let $\{\bar{w}^{\d(\e)}(t)\}_{t\geq 0}$ be a  Wiener process, with covariance $Q_{\d(\e)}$, defined on $(\bar{\Omega}, \bar{\mathcal{F}},\bar{\mathbb{P}})$ and corresponding to the filtration $\{\bar{\mathcal{F}}_t\}_{t\geq 0}$. Moreover, let $\{\bar{\varphi}_\e\}_{\e>0}$ and $\bar{\varphi}$ be  $\{\bar{\mathcal{F}}_t\}_{t\geq 0}$-predictable processes in $\mathcal{A}^\gamma_T$, such that  the distribution of $(\bar{\varphi}_{\e}, \bar{\varphi}, \bar{w}^{\,\d(\e)})$ coincides with the distribution of $(\varphi_\e,\varphi,w^{\,\d(\e)})$ and
 \[\lim_{\e\to 0}\bar{\varphi}_\e=\bar{\varphi}\ \ \ \text{weakly in } L^2(0,T;H),\ \ \ \ \bar{\mathbb{P}}-\text{a.s.}\]
 Then, if $\bar{u}^{\,\bar{\varphi}_\e}_\e$ is the solution of an equation analogous to \eqref{rd49}, with $\varphi_\e$ and $w^{\,\d(\e)}$ replaced respectively by $\bar{\varphi}_\e$ and $\bar{w}^{\,\d(\e)}$, we have that 
 \[\lim_{\e\to 0} |\bar{u}_\e^{\bar{\varphi}_\e}-\bar{u}^{\bar{\varphi}}|_{\mathcal{E}}=0,\ \ \ \ \ \mathbb{P}-\text{a.s.}\]
 where $\mathcal{E}=C([0,T];H^{-s}(D))$ if \eqref{rd46} holds and 
$\mathcal{E}=C([0,T];H)$ if \eqref{rd5050} holds.
In what follows, when proving the above statement, we will just forget about the $-$.

\subsection{Proof of Theorem \ref{teo5.1}}

In fact, we only need  to prove Condition 2, introduced above. Actually, we will see that Condition 1 follows from the same arguments, as a special case.

To this purpose, we fix a sequence $\{\varphi_\e\}_{\e>0} \subset \mathcal{A}^\gamma_T$ which is $\mathbb{P}$-a.s. convergent to some $\varphi \in\,\mathcal{A}^\gamma_T$, with respect to the weak topology of $L^2(0,T;H)$, and we denote by $u^{x, \varphi_\e}_\e$ the solution of  equation \eqref{rd49} starting from the initial condition $x \in\,H$. Our purpose is showing that, if $u^{x,\varphi}$ is the solution of equation \eqref{rd50}, then
 \begin{equation}
 \label{rd71}
\lim_{\e\to 0} \mathbb{E}\,\le|u^{\,{x, \varphi}_\e}_\e-{u}^{\,{x, \varphi}}\r|_{C([0,T];H^{-s}(D))}=0,
\end{equation}
or
\begin{equation}
 \label{rd71-bis}
 \lim_{\e\to 0} \mathbb{E}\,\le|u^{\,{x, \varphi}_\e}_\e-{u}^{\,{x, \varphi}}\r|_{C([0,T];H)}=0,
\end{equation}
depending on the different scaling conditions between $\e$ and $\d(\e)$ that we assume in Theorem \ref{teo5.1}.
In fact, to prove Condition 2, we would just need $\mathbb{P}$-almost sure convergence.

Before proving \eqref{rd71} or \eqref{rd71-bis}, we introduce some notation and prove a preliminary result.
For every $\varphi \in\,L^2(0,T;H)$, we  define
\[\Phi(\varphi)(t):=\int_0^t e^{(t-s)A}\varphi(s)\,ds.\]
As shown e.g. in \cite[Proposition A.1.]{dpz-erg}, 
for every $\gamma<1/2$ 
\begin{equation}
\label{rd55}
\Phi:L^2(0,T;H)\to C^{\frac 12-\gamma}([0,T];D((-A)^\gamma),\end{equation}
is a bounded linear operator. 
In particular, due to the   continuity of mapping \eqref{rd55} and to the compactness of the  embedding  
$C^{1/2-\gamma}([0,T];D((-A)^\gamma))\hookrightarrow C([0,T];H),$
if $\{\varphi_k\}_{k \in\,\mathbb{N}}$ is a bounded sequence in $L^2(0,T;H)$, weakly convergent to some $\varphi \in\,L^2(0,T;H)$,
we have
\begin{equation}
\label{rd56}
\lim_{k\to\infty}\le|\Phi(\varphi_k)-\Phi(\varphi)\r|_{C([0,T];H)}=0.\end{equation}
Next, for every $\e>0$ and $\varphi \in\,L^2(0,T;H)$, we  define
\[\Phi_\e(\varphi)(t):=\int_0^t e^{(t-s)A}Q_{\d(\e)}\varphi(s)\,ds=\Phi(Q_{\d(\e)}\varphi)(t).\]

\begin{Lemma}
\label{l5.5bis}
If  $\{\varphi_\e\}_{\e>0}$ is a family of processes in $\mathcal{A}_T^\gamma$ that converges almost surely, as $\e\downarrow 0$, to some $\varphi \in\,\mathcal{A}_T^\gamma$, in the space $L^2(0,T;H)$, endowed with the weak topology, then
\begin{equation}
\label{rd57bis}
\lim_{\e\to 0}\le|\Phi_\e(\varphi_\e)-\Phi(\varphi)\r|_{C([0,T];H)}=0,\ \ \ \ \mathbb{P}-\text{a.s}.
\end{equation}
\end{Lemma}
\begin{proof}
For every $\e>0$, we have
\begin{equation}
\label{rd70}
\Phi_\e(\varphi_\e)-\Phi(\varphi)=\Phi(Q_{\d(\e)}(\varphi_\e-\varphi))+\Phi(Q_{\d(\e)}\varphi-\varphi).\end{equation}
Since 
\[\lim_{\e\to 0} Q_{\d(\e)}(\varphi_\e-\varphi)=0,\ \ \ \ \text{weakly in}\ L^2(0,T;H),\]
with $Q_{\d(\e)}(\varphi_\e-\varphi) \in\,\mathcal{A}_T^\gamma$, and $Q_{\d(\e)}\varphi$ converges to $\varphi$ in $L^2(0,T;H)$, as $\e\to 0$, our lemma follows from \eqref{rd56} and from the continuity of the mapping $\Phi:L^2(0,T;H)\to C([0,T];H)$.
\end{proof}

Now, we can proceed with the proof of  \eqref{rd71} and \eqref{rd71-bis}. From now on, $x \in\,H$ is the fixed initial condition in the statement of Theorem \ref{teo5.1} and $y \in\,H^1_0(D)$ is some other initial condition to be determined later on.
For every   $\e>0$, we define
\begin{equation}
\label{fin7}\rho_{1}^{\e}(t):=u^{x,\varphi_\e}_\e(t)-u^{y,\varphi_\e}_\e(t),
\ \ \ \ t \in\,[0,T].\end{equation}
We have
\[\frac{d\rho_{1}^{\e}}{dt}(t)=A\rho_{1}^{\e}(t)+\le[ F(u^{x,\varphi_\e}_\e(t))-F(u^{y,\varphi_\e}_\e(t))\r],\ \ \ \ \rho_{1}^{\e}(0)=x-y,\]
so that, thanks to \eqref{rd3}, we get
\begin{equation}
\label{fin1}
|u^{x,\varphi_\e}_\e(t)-u^{y,\varphi_\e}_\e(t)|_H^2=|\rho_{1}^{\e}(t)|^2_H\leq e^{\la_1 t}|x-y|_H^2.\end{equation}
In the same way, if we define
\begin{equation}
\label{fin50}
\rho(t):=u^{y,\varphi}(t)-u^{x,\varphi}(t),\ \ \ \ \ t \in\,[0,T],
\end{equation}
we get
\begin{equation}
\label{fin2}
|u^{y,\varphi}(t)-u^{x,\varphi}(t)|_H^2=|\rho(t)|_H^2 \leq e^{\la_1 t}|x-y|_H^2.\end{equation}

Now, for every $\e>0$, we define
\begin{equation}
\label{fin8}
\vartheta_\e(t):=u^{y,\varphi_\e}_\e(t)-\sqrt{\e}z_{\d(\e)}(t),\ \ \ \ t \in\,[0,T],\end{equation}
where, for every $\d>0$, $z_\d(t)$ is the solution to problem \eqref{linear},  that is
\[z_{\d}(t)=\int_0^t e^{(t-s)A}\,dw^\d(s),\ \ \ \ t \in\,[0,T].\]
This means that
\[\frac{d\vartheta_\e}{dt}(t)=A \vartheta_\e(t)+F(\vartheta_\e(t)+\sqrt{\e}\,z_{\d(\e)}(t))+Q_{\d(\e)}\varphi_\e(t),\ \ \ \ \vartheta_\e(0)=y,\]
so that, thanks to \eqref{fin4} and \eqref{rd3}
\[\begin{array}{l}
\ds{\frac 12\,\frac d{dt} |\vartheta_\e(t)|^2_H+|\vartheta_\e(t)|^2_{H^1}=
\langle F(\vartheta_\e(t)+\sqrt{\e}\,z_{\d(\e)}(t))-F(\sqrt{\e}\,z_{\d(\e)}(t)),\vartheta_\e(t)\rangle _{q_n, p_n}}\\
\vs
\ds{+\langle F(\sqrt{\e}\,z_{\d(\e)}(t)),\vartheta_\e(t)\rangle _{q_n, p_n}+\langle Q_{\d(\e)}\varphi_\e(t),\vartheta_\e(t)\rangle_H}\\
\vs
\ds{\leq -c\,|\vartheta_\e(t)|_{p_n}^{p_n}+\la_1|\vartheta_\e(t)|_H^2+|F(\sqrt{\e}\,z_{\d(\e)}(t))|_{q_n}|\vartheta_\e(t)|_{p_n}+c\,|\varphi_\e(t)|_H\,|\vartheta_\e(t)|_H}\\
\vs
\ds{\leq -\frac c2 \,|\vartheta_\e(t)|_{p_n}^{p_n} + c\,|\sqrt{\e}\,z_{\d(\e)}(t))|_{p_n}^{ 2n+2}+c\,|\vartheta_\e(t)|_H^2+c\,|\varphi_\e(t)|_H^2.  }
\end{array}\]
As a consequence of the Gronwall Lemma, since $\varphi_\e \in\,\mathcal{A}^\gamma_T$, this implies
\[|\vartheta_\e(t)|^2_H+\int_0^t |\vartheta_\e(s)|^2_{H^1}\,ds+\int_0^t |\vartheta_\e(s)|_{p_n}^{p_n}\,ds\leq c_T\,\le(|y|_H^2+\int_0^t |\sqrt{\e}\,z_{\d(\e)}(s))|_{p_n}^{ 2n+2}\,ds+\gamma\r),\]
and then, due to \eqref{rd30}, we conclude that for every $\la>0$
\begin{equation}
\label{fin5}
\begin{array}{l}
\ds{
\mathbb{E}\le(\sup_{t \in\,[0,T]}|\vartheta_\e(t)|_H^2+\int_0^T |\vartheta_\e(s)|_{H^1(D)}^2\,ds+\int_0^t |\vartheta_\e(s)|_{p_n}^{p_n}\,ds\r)^\la}\\
\vs
\ds{\leq c_{\la}(T)\,\le(|y|_H^{2\la}+\gamma^\la\r)+c_{\la}(T)\le[\e\, \La(\d(\e))\r]^{\la(n+1)}.}
\end{array}
\end{equation}

Next, we define
\begin{equation}
\label{fin10}
\rho_2^\e(t)=\vartheta_\e(t)-u^{y,\hat{\varphi}_\e}(t),\ \ \ \ t \in\,[0,T],\end{equation}
where $\vartheta_\e(t)$ is the process defined in \eqref{fin8} and $
\hat{\varphi}_\e=Q_{\d(\e)}\varphi_\e$. We have that $\rho_2^\e(t)$ satisfies the equation
\[\frac{d \rho_2^\e}{dt}(t)=A\rho_2^\e(t)+F(\vartheta_\e(t)+\sqrt{\e}\,z_{\d(\e)}(t))-F(u^{y,\hat{\varphi}_\e}(t)),\ \ \ \ \rho_2^\e(0)=0,\]
and then
\[\begin{array}{l}
\ds{\frac 12\,\frac d{dt} |\rho_2^\e(t)|^2_H+|\rho_2^\e(t)|^2_{H^1}}\\
\vs
\ds{=
\langle F(\vartheta_\e(t)+\sqrt{\e}\,z_{\d(\e)}(t))-F(u^{y,\hat{\varphi}_\e}(t)),\rho_2^\e(t)+\sqrt{\e}\,z_{\d(\e)}(t)\rangle _{q_n,p_n}}\\
\vs
\ds{-\langle F(\vartheta_\e(t)+\sqrt{\e}\,z_{\d(\e)}(t))-F(u^{y,\hat{\varphi}_\e}(t)),\sqrt{\e}\,z_{\d(\e)}(t)\rangle _{q_n, p_n}.}\end{array}\]
According to \eqref{rd10} and \eqref{rd3}, this implies
\[\begin{array}{l}
\ds{\frac 12\,\frac d{dt} |\rho_2^\e(t)|^2_H+|\rho_2^\e(t)|^2_{H^1}\leq c\,|\rho_2^\e(t)|_H^2+c\,|\sqrt{\e}\,z_{\d(\e)}(t)|_H^2}\\
\vs
\ds{+\le(|\vartheta_\e(t)|_{p_n}^{2n+1}+|\sqrt{\e}\,z_{\d(\e)}(t)|_{p_n}^{2n+1}+|u^{y,\hat{\varphi}_\e}(t)|_{p_n}^{2n+1}\r)|\sqrt{\e}\,z_{\d(\e)}(t)|_{p_n},}
\end{array}\]
and from the Gronwall lemma we obtain
\[\begin{array}{l}
\ds{\sup_{t \in\,[0,T]}|\rho_2^\e(t)|^2_H\leq c_T\,\int_0^T|\sqrt{\e}\,z_{\d(\e)}(t)|_H^2\,dt}\\
\vs
\ds{+c_T\,\int_0^T\le(|\vartheta_\e(t)|_{p_n}^{2n+1}+|\sqrt{\e}\,z_{\d(\e)}(t)|_{p_n}^{2n+1}+|u^{y,\hat{\varphi}_\e}(t)|_{p_n}^{2n+1}\r)\,|\sqrt{\e}\,z_{\d(\e)}(t)|_{p_n}\,dt.}
\end{array}\]
By taking the expectation of both sides, this yields
\[\begin{array}{l}
\ds{\mathbb{E}\sup_{t \in\,[0,T]}|\rho_2^\e(t)|^2_H\leq c_T\,\sup_{t \in\,[0,T]}\mathbb{E}\,|\sqrt{\e}\,z_{\d(\e)}(t)|_H^2+c_T\,\sup_{t \in\,[0,T]}\le(\mathbb{E}\,|\sqrt{\e}\,z_{\d(\e)}(t)|_{p_n}^{p_n}\r)^{\frac 1{p_n}} }\\
\vs
\ds{\times  \le(\mathbb{E}\int_0^T |\vartheta_\e(t)|_{p_n}^{p_n}\,dt+\sup_{t \in\,[0,T]}\,\mathbb{E}\,|\sqrt{\e}\,z_{\d(\e)}(t)|_{p_n}^{p_n}+\int_0^T |u^{y,\hat{\varphi}_\e}(t)|_{p_n}^{p_n}\,dt\r)^{\frac 1{q_n}},}
\end{array}\]
and  since
\[|\hat{\varphi}_\e|_{L^2(0,T;H)}=|Q_\e\varphi_\e|_{L^2(0,T;H)}\leq \sqrt{\gamma},\]
thanks to \eqref{rd46}, \eqref{fin5}, \eqref{rd30} and \eqref{rd20}, we conclude that for every $\e \in\,(0,1]$
\begin{equation}
\label{rd60}
\begin{array}{l}
\ds{\mathbb{E}\sup_{t \in\,[0,T]}|\rho_2^\e(t)|^2_H\leq c_T\le(1+|y|_H^{2n+1} + \gamma^{\frac{2n+1}2}  \r) \le[\e\, \La(\d(\e))\r]^{\frac 12}.}
\end{array}
\end{equation}

Finally, we define
\begin{equation}
\label{rd61}
\rho_3^\e(t)=u^{y,\hat{\varphi}_\e}(t)-u^{y,\varphi}(t),\ \ \ \ \ t \in\,[0,T].
\end{equation}
We have
\[\begin{array}{l}
\ds{\frac 12\frac d{dt}|u^{y,\hat{\varphi}_\e}(t)|^2_{H^1(D)}+|Au^{y,\hat{\varphi}_\e}(t)|_H^2}\\
\vs
\ds{=\langle F(u^{y,\hat{\varphi}_\e}(t))-F(0),A u^{y,\hat{\varphi}_\e}(t)\rangle_H+\langle F(0),A u^{y,\hat{\varphi}_\e}(t)\rangle_H+\langle \hat{\varphi}_\e(t),Au^{y,\hat{\varphi}_\e}(t)\rangle_H,}
\end{array}
\]
so that
\begin{equation}
\label{rd63}
\begin{array}{l}
\ds{ \frac d{dt}|u^{y,\hat{\varphi}_\e}(t)|^2_{H^1(D)}+|Au^{y,\hat{\varphi}_\e}(t)|_H^2}\\
\vs
\ds{\leq \langle F(u^{y,\hat{\varphi}_\e}(t))-F(0),A u^{y,\hat{\varphi}_\e}(t)\rangle_H+c\,| \hat{\varphi}_\e(t)|^2_H+c.}
\end{array}
\end{equation}
If we assume $y \in\,H^1_0(D)$, integrating by parts we have
\[\begin{array}{l}
\ds{\langle F(u^{y,\hat{\varphi}_\e}(t))-F(0),A u^{y,\hat{\varphi}_\e}(t)\rangle_H=\int_D f^\prime(u^{y,\hat{\varphi}_\e}(t,x)) |\nabla u^{y,\hat{\varphi}_\e}(t,x)|^2\,dx}\\
\vs
\ds{\leq -c\,\int_D|u^{y,\hat{\varphi}_\e}(t,x)|^{2n}|\nabla u^{y,\hat{\varphi}_\e}(t,x)|^2\,dx+c\,|u^{y,\hat{\varphi}_\e}(t)|_{H^1(D)}^2.}
\end{array}\]
Therefore, due to \eqref{rd63} we obtain
\[\frac d{dt}|u^{y,\hat{\varphi}_\e}(t)|^2_{H^1(D)}+|Au^{y,\hat{\varphi}_\e}(t)|_H^2\leq c\,|u^{y,\hat{\varphi}_\e}(t)|_{H^1(D)}^2+
c\,| \hat{\varphi}_\e(t)|^2_H+c,\]
which implies
\[\sup_{t \in\,[0,T]}|u^{y,\hat{\varphi}_\e}(t)|^2_{H^1(D)}+\int_0^T|Au^{y,\hat{\varphi}_\e}(t)|_H^2\,dt\leq c\le(|y|_{H^1(D)}^2+\gamma+1\r).\]
This means that the family
\begin{equation}
\label{rd64}
\{u^{y,\hat{\varphi}_\e}\}_{\e >0}\subset C([0,T];H^1_0(D))\cap L^2(0,T;D(A))
\end{equation}
is $\mathbb{P}$-a.s. bounded. Moreover, according to \eqref{rd20}, we have that the family $\{u^{y,\hat{\varphi}_\e}\}_{\e >0}$ is bounded in $L^{p_n}((0,T)\times D)$, so that
\[\{F(u^{y,\hat{\varphi}_\e})\}_{\e >0}\subset L^{q_n}((0,T)\times D)\hookrightarrow  L^{q_n}(0,T;(W^{\eta,2}(D))^\prime(D)),\ \ \ \ \eta:=\frac{dn}{p_n},\]
is bounded. In particular, we obtain that 
\[\{u^{y,\hat{\varphi}_\e}\}_{\e >0}\subset W^{1,q_n}(0,T;(W^{\eta,2}(D))^\prime(D))\hookrightarrow C^\a([0,T];(W^{\eta,2}(D))^\prime(D)),\ \ \ \ \ \a<\frac 1{p_n},\]
is bounded.
This, together with \eqref{rd64}, implies that
\[\{u^{y,\hat{\varphi}_\e}\}_{\e >0}\subset C([0,T];H)\ \text{is compact}.\]
As a consequence of Lemma \ref{l5.5bis}, any limit point of  $\{u^{y,\hat{\varphi}_\e}\}_{\e >0}$ has to coincide with $u^{y,\varphi}$, so that we can conclude that
\[
\lim_{\e\to 0}\sup_{t \in\,[0,T]}\,|\rho_3^\e(t)|_{H}=\lim_{\e\to 0}\sup_{t \in\,[0,T]}\,|u^{y,\hat{\varphi}_\e}(t)-u^{y,\varphi}(t)|_{H}=0,\ \ \ \ \ \mathbb{P}-\text{a.s.}\]
Moreover, due to \eqref{rd20}, the family 
$\{\sup_{t \in\,[0,T]}\,|\rho_3^\e(t)|_{H}\}_{\e>0}\subset L^1(\Omega)$
is equi-integrable, so that for every fixed $y \in\,H^1_0(D)$
\begin{equation}
\label{rd100}
\lim_{\e\to 0}\mathbb{E}\sup_{t \in\,[0,T]}\,|\rho_3^\e(t)|_{H}=0.
\end{equation}

Now, collecting all terms defined above in \eqref{fin7}, \eqref{fin50}, \eqref{fin8}, \eqref{fin10}, \eqref{rd61}, we have
\[u^{\,{x, \varphi}_\e}_\e(t)-{u}^{\,{x, \varphi}}(t)=\sum_{i=1}^3\rho_i^\e(t)+\rho(t)+\sqrt{\e}\,z_{\d(\e)}(t),\ \ \ \ \ \ t \in\,[0,T].\]
Thanks to \eqref{fin1}, \eqref{fin2} and \eqref{rd60}, this implies
\[\begin{array}{l}
\ds{\mathbb{E}\,\sup_{t \in\,[0,T]}\,|u^{\,{x, \varphi}_\e}_\e(t)-{u}^{\,{x, \varphi}}(t)|_X}\\
\vs
\ds{\leq c_T\,|x-y|_H+c_{T,\gamma}\le(1+|y|_H^{2n+1}\r) \le[\e\, \La(\d(\e))\r]^{\frac 12}+\mathbb{E}\sup_{t \in\,[0,T]}\,|\rho_3^\e(t)|_{H}+\sqrt{\e}\,\mathbb{E}\sup_{t \in\,[0,T]}\,|z_{\d(\e)}(t)|_X,}
\end{array}\]
where $X=H$ or $X=H^{-s}(D)$.
For an arbitrary $\rho>0$, we fix $y \in\,H^1_0(D)$ such that $c_T\,|x-y|_H<\rho$. Therefore, from \eqref{rd46}, \eqref{rd5050}, \eqref{rd100}, \eqref{rd43} and \eqref{rd43-bis}, we get
\[\limsup_{\e\to 0}\mathbb{E}\,\sup_{t \in\,[0,T]}\,|u^{\,{x, \varphi}_\e}_\e(t)-{u}^{\,{x, \varphi}}(t)|_X\leq \rho,\]
and, due to the arbitrariness of $\rho$, we conclude that \eqref{rd71} and \eqref{rd71-bis} hold.

\appendix

\section{Appendix}
\label{appendixA}

For every $\d>0$ and $\theta \in\,(0,1)$, we denote
\begin{equation}
\label{rd40}
z_{\d,\theta}(t)=\int_0^t (t-s)^{-\frac \theta 2} e^{(t-s)A} dw^\d(s),\ \ \ \ \ t\geq 0.\end{equation}
In case $\theta=0$, we denote $z_{\d,0}(t)=z_\d(t)$.

\begin{Lemma}
\label{LA.1}
Under 
Hypotheses \ref{H1} and \ref{H2}, there exists $\bar{\theta} \in\,(0,1)$ such that for any $\kappa, p\geq 1$ and $T>0$ and for any $\d \in\,(0,1)$ and $\theta \in\,[0,\bar{\theta})$
  we have
\begin{equation}
\label{rd30}
\sup_{t \in\,[0,T]}\,\mathbb{E} |z_{\d,\theta}(t)|^\kappa_{p}\leq c_{\kappa,p}(T)\,\La_\theta(\d)^{\frac{\kappa}2},
\end{equation}
where
\[\La_\theta(\d)=\begin{cases}
\ds{\log \d^{-1},}  &  \text{if}\ \a=\theta=0,\ d=2,\\
\vs
\ds{\d^{-(d-2(1-\theta)+\a)},}   &  \text{otherwise.}
\end{cases}\]

\end{Lemma}

\begin{proof}
According to \eqref{rd1}, for every $p\geq 2$ we have
\[\begin{array}{l}
\ds{\mathbb{E} |z_{\d,\theta}(t)|^p_{p}=\mathbb{E}\int_D\le|\,\sum_{k=1}^\infty \int_0^t(t-r)^{-\frac \theta 2}e^{-(t-r)\a_k}\la_k(\d) e_k(\xi)\,d\beta_k(r)\r|^p\,d\xi}\\
\vs
\ds{\leq c_p\int_D\le(\int_0^t (t-r)^{-\theta}\sum_{k=1}^\infty e^{-2(t-r)\a_k}\la_k^2(\d) |e_k(\xi)|^2\,dr\r)^{\frac p2}\,d\xi}\\
\vs
\ds{\leq c_p |D| \le(\,\sum_{k=1}^\infty\la_k^2(\d) k^{\frac {\a}d}\int_0^t r^{-\theta} e^{-2r \a_k}\,dr\r)^{\frac p2}\leq c_p |D| \le(\,\sum_{k=1}^\infty\frac{\la_k^2(\d)\, k^{\frac {\a}d}}{\a_k^{1-\theta}}\r)^{\frac p2}.}
\end{array}\]
Hence, thanks to \eqref{rd22} and \eqref{rd31}, we obtain,
\begin{equation}
\label{rd36}
\mathbb{E} |z_{\d,\theta}(t)|^p_{p}\leq c_p(T)\,|D|\le(\,\sum_{k=1}^\infty \frac{1}{k^{\frac{2(1-\theta)-\a}{d}}(1+\d k^{\frac{1}d})^{2\beta}}\r)^{\frac p2}.
\end{equation}
Notice that, due to \eqref{rd35}, there exists $\bar{\theta} \in\,(0,1)$ such that the series above is convergent, for every fixed $\d>0$ and $\theta \in\,[0,\bar{\theta})$.

We have
\[\sum_{k=1}^\infty \frac{1}{k^{\frac{2(1-\theta)-\a}{d}}(1+\d k^{\frac{1}d})^{2\beta}}\sim \int_1^\infty \frac{1}{x^{\frac{2(1-\theta)-\a}{d}}(1+\d x^{\frac{1}d})^{2\beta}}\,dx,\]
and then, with a change of variable, we obtain
\[\sum_{k=1}^\infty \frac{1}{k^{\frac{2(1-\theta)-\a}{d}}(1+\d k^{\frac{1}d})^{2\beta}}\sim d\, \d^{-(d-2(1-\theta)+\a)}\int_\d^\infty\frac{1}{x^{1-(d-2(1-\theta)+\a)}(1+x)^{2\beta}}\,dx.
\]
Therefore, if  $\a=\theta=0$ and $d=2$, since $\beta>0$ we have
\[\sum_{k=1}^\infty \frac{1}{k^{\frac{2(1-\theta)-\a}{d}}(1+\d k^{\frac{1}d})^{2\beta}}\sim c\,\log \frac 1\d.\]
Otherwise, according to Hypothesis \ref{H2}, there exists $\bar{\theta}>0$ such that
\[2\beta-(d-2(1-\theta)+\a)>0,\]
for every $\theta \in\,[0,\bar{\theta})$. Hence, as $d-2(1-\theta)+\a>0$, we get
\[\begin{array}{l}
\ds{\sum_{k=1}^\infty \frac{1}{k^{\frac{2(1-\theta)-\a}{d}}(1+\d k^{\frac{1}d})^{2\beta}}}\\
\vs
\ds{\sim d\, \d^{-(d-2(1-\theta)+\a)}\int_0^\infty\frac{1}{x^{1-(d-2(1-\theta)+\a)}(1+x)^{2\beta}}\,dx\leq c\,\d^{-(d-2(1-\theta)+\a)}.}
\end{array}\]
This implies \eqref{rd30}, in case $\kappa=p$. The general case follows from the H\"older inequality.

\end{proof}

Next, for every $s>0$, we have
\[\mathbb{E}\,|z_{\d,\theta}(t)|^2_{H^{-s}(D)}=\sum_{k=1}^\infty \int_0^t (t-r)^{-\theta} e^{-2(t-r)\a_k}\la_k^2(\d)\a_k^{-s}\,dr.\]
Therefore, by proceeding as in the proof of Lemma \ref{LA.1}
we conclude

\begin{Lemma}
\label{LA.2}
Under 
Hypotheses \ref{H1} and \ref{H2}, there exists $\bar{\theta} \in\,(0,1)$ such that for any $s, T>0$ and for any $\d \in\,(0,1)$ and $\theta \in\,[0,\bar{\theta})$
  we have
\begin{equation}
\label{rd5000}
\sup_{t \in\,[0,T]}\,\mathbb{E}\, |z_{\d,\theta}(t)|^2_{H^{-s}(D)}\leq c(T)\,\Gamma_{\theta,s}(\d),
\end{equation}
where
\[\Gamma_{\theta,s}(\d)=\begin{cases}
\ds{\log \d^{-1},}  &  \text{if}\ \theta=s,\ d=2,\\
\vs
\ds{\d^{-(d-2(1-\theta)-2s)},}   &  \text{otherwise.}
\end{cases}\]

\end{Lemma}

Now, let us consider the linear problem 
\begin{equation}
\label{linear}
dz(t)=Az(t)\,dt+dw^\d(t),\ \ \ \ z(0)=0.
\end{equation}
Its unique mild solution $z_\d(t)$ coincides with the process $z_{\d,0}(t)$ defined in \eqref{rd40}, for $\theta=0$. Notice that, due to \eqref{rd30}, we have
\begin{equation}
\label{rd41}
\sup_{t \in\,[0,T]}\,\mathbb{E} |z_{\d}(t)|^\kappa_{L^p(D)}\leq c_{\kappa,p}(T)|D|\begin{cases}
\le(\log \d^{-1}                                                                                                                                                                         \r)^{\frac \kappa2},  &  \text{if $\a=0$ and $d=2$,}\\
\vspace{.1mm}  &\\
\d^{-\frac \kappa 2(d-2+\a)},  &  \text{otherwise.}
\end{cases}
\end{equation}

By using a stochastic factorization argument, for every $\theta \in\,(0,1)$, we have
\[z_\d(t)=\frac{\sin (\theta \pi)}{2\pi}\int_0^t (t-\si)^{\frac \theta 2-1}e^{(t-\si)A}z_{\d,\theta}(\si)\,d\si.\]

If we take $\kappa>2/\theta$, we have
\[|z_\d(t)|^\kappa_H\leq c_{\kappa,\theta}\le(\int_0^T \si^{\le(\frac \theta 2-1\r)\frac{\kappa}{\kappa-1}}\,d\si\r)^{\kappa-1}\int_0^t|z_{\d,\theta}(\si)|_H^\kappa\,d\si\leq c_{\kappa,\theta}(T)\int_0^t|z_{\d,\theta}(\si)|_H^\kappa\,d\si.\]
Therefore, if we fix $\gamma>d-2+\a$ and 
we pick $ \theta_\gamma \in\,(0,\bar{\theta})$ such that 
\[d-2(1-\theta_\gamma)+\a<\gamma,\]
thanks to \eqref{rd30}, we get
\begin{equation}
\label{rd45}
\mathbb{E}\,\sup_{t \in\,[0,T]} |z_\d(t)|^\kappa_H\leq c_{\kappa,\si}(T)\,\d^{-\frac{\gamma \kappa}2}.\end{equation}

Thus, we have proven the following result.

\begin{Lemma}
\label{LA2}
Under Hypotheses \ref{H1} and \ref{H2}, for every $\kappa\geq 2$ and $\d>0$ we have that  for every
 \[\gamma>d-2+\a,\]
it holds
\begin{equation}
\label{rd43}
\mathbb{E}\sup_{t \in\,[0,T]} |z_\d(t)|^\kappa_H\leq c_{\kappa,\gamma}(T)\, \d^{- \frac {\gamma \kappa}2},\ \ \ \ \d \in\,(0,1).
\end{equation}

\end{Lemma}

Finally, by using again a factorization argument,  for every $s>0$ and $\kappa>\frac 2s\vee 1$ we have
\[|z_\d(t)|^\kappa_{H^{-s}(D)}\leq c\,\le(\int_0^T \si^{-\le(\frac s 2-1\r)\frac \kappa{\kappa-1}}\,d\si\r)^{\kappa-1}\int_0^t |z_{\d,s}(\si)|_{H^{-\rho}(D)}^\kappa\,d\si\leq c(T)\int_0^t |z_{\d,s}(\si)|_{H^{-s}(D)}^\kappa\,d\si.\]

Therefore, due to \eqref{rd5000} we can conclude that the following result is true.

\begin{Lemma}
\label{LA3}
Under Hypotheses \ref{H1} and \ref{H2}, for every $s>0$, $\d \in\,(0,1)$ and $\kappa\geq 1$we have that \begin{equation}
\label{rd43-bis}
\mathbb{E}\sup_{t \in\,[0,T]} |z_\d(t)|^\kappa_{H^{-s}(D)}\leq c_{\rho}(T)\begin{cases}
\log \d^{-1},\ \ \ \ \text{if}\ d=2,\\
\vs
\d^{-(d-2)},\ \ \ \ \text{if}\ d\geq 3.
\end{cases}.
\end{equation}

\end{Lemma}


\end{document}